\documentclass{article}

\usepackage{graphicx}
\usepackage[utf8]{inputenc}
\usepackage[T1]{fontenc}
\usepackage[english]{babel}
\usepackage{amssymb,amsthm,amsmath,amstext,amscd}
\usepackage{xcolor}
\usepackage[normalem]{ulem}
\usepackage{tikz-cd}
\usepackage{indentfirst}
\usepackage{hyperref}
\usepackage{verbatim}
\usepackage{enumitem}

\allowdisplaybreaks
\numberwithin{equation}{section}
\theoremstyle{plain}

\newtheorem{theorem}{Theorem}

\newtheorem{corollary}[theorem]{Corollary}
\newtheorem{proposition}[theorem]{Proposition}
\newtheorem{lemma}[theorem]{Lemma}

\newtheorem{remark}[theorem]{Remark}

\newtheorem{problem}[theorem]{Problem}

\newcommand{\der}{{\rm Der}}

\newcommand{\degt}{\deg_t}
\newcommand{\Span}{\operatorname{span}}
\newcommand{\Z}{\operatorname{Z}}
\DeclareMathOperator{\LND}{LND}
\DeclareMathOperator{\LFD}{LFD}

\DeclareMathOperator{\Aut}{Aut}

\DeclareMathOperator{\LFDer}{LFD}
\DeclareMathOperator{\ML}{ML}
\DeclareMathOperator{\ad}{ad}
\DeclareMathOperator{\id}{id}

\title{A Characterization of Local Nilpotence for Derivations of Ore Extensions}
\author{Rene Baltazar, Samuel A. Lopes and Oscar Morales}
\date{}

\begin{document}

\maketitle

\begin{abstract}

Let $\Bbbk$ be an algebraically closed field of characteristic zero. A recent theorem of I. Pan characterizes the nonzero locally nilpotent derivations of $\Bbbk[x,y]$ in terms of their isotropy groups: such a
derivation is locally nilpotent if and only if its isotropy group contains automorphisms of arbitrarily large degree. We study analogues of this characterization for several noncommutative algebras related to the affine plane. Our main result shows that Pan's characterization extends to the differential Ore extensions $A_h=\Bbbk[x][t;h(x)\partial_x]$, where $h\in\Bbbk[x]\setminus\Bbbk$. For the first Weyl algebra $A_1$, we establish the same equivalence for every nonzero locally finite derivation. For the quantum plane and the first quantum Weyl algebra, known results imply that the corresponding equivalence holds vacuously for nonzero derivations. In contrast, the converse fails for the free associative algebra $\Bbbk\langle x,y\rangle$: we construct a non-locally-nilpotent inner derivation whose isotropy group has unbounded degree. We further analyze this failure through the behavior of automorphisms and derivations under abelianization.

\end{abstract}

\section{Introduction}\label{section1}

Let $\Bbbk$ be an algebraically closed field of characteristic zero. For a $\Bbbk$-algebra $A$, we denote by $\Aut(A)$ and $\der(A)$ the group of $\Bbbk$-automorphisms and the space of $\Bbbk$-derivations of
$A$, respectively. The group $\Aut(A)$ acts by conjugation on $\der(A)$ and, for $D\in\der(A)$, the \emph{isotropy group of $D$} is 
\[
\Aut_D(A)=\{\rho\in\Aut(A)\mid \rho D=D\rho\}.
\]

A derivation $D$ is said to be \emph{locally finite} if, for every $a\in A$, the $\Bbbk$-vector space $\Span_{\Bbbk}\{D^n(a)\mid n\geq0\}$ is finite-dimensional. It is said to be \emph{locally nilpotent} if,
for every $a\in A$, there exists $n\geq1$ such that $D^n(a)=0$. Thus, every locally nilpotent derivation is locally finite. We denote by $\LFD(A)$ and $\LND(A)$ the sets of locally finite and locally nilpotent derivations of $A$, respectively.

A classical theorem of Rentschler \cite{Re1968} states that every locally nilpotent derivation of $\Bbbk[x,y]$ is, after a suitable polynomial change of coordinates, of the form $f(x)\partial_y$, where $f(x)\in\Bbbk[x]$. Locally finite derivations of $\Bbbk[x,y]$ are also classified, up to conjugation by polynomial automorphisms; see \cite[Corollary 4.7]{V92}. For example, Van den Essen used this classification to obtain an algorithm for deciding whether a polynomial vector field in dimension two has a polynomial flow; see \cite[Section 5.]{V92}.

A central motivation for this work is the following theorem of I. Pan \cite{Pan2022}, which characterizes local nilpotence of plane polynomial derivations in
terms of their isotropy groups. The proof of this characterization relies strongly on earlier work of R. Baltazar and I. Pan \cite{BaltazarPan2021} concerning the structure of isotropy groups of polynomial derivations.

\begin{theorem}[{\cite[Theorem 1.1]{Pan2022}}]\label{thm_Pan}
Let $D$ be a nonzero derivation of $\Bbbk[x,y]$. The following conditions are equivalent:
\begin{enumerate}[
    label=\textnormal{(\alph*)},
    itemsep=2pt,
    topsep=3pt
]
    \item $D$ is locally nilpotent;
    \item for every $d\geq1$, there exists
    $\rho=(f,g)\in\Aut_D(\Bbbk[x,y])$ such that
    $\deg(f)\geq d$ or $\deg(g)\geq d$;
    \item $\Aut_D(\Bbbk[x,y])$ is not an algebraic subgroup of
    $\Aut(\Bbbk[x,y])$;
    \item $\Aut_D(\Bbbk[x,y])$ is an infinite-dimensional algebraic
    subgroup of $\Aut(\Bbbk[x,y])$ in the sense of
    Shafarevich--Kambayashi.
\end{enumerate}
\end{theorem}

A natural question suggested by Pan's theorem is whether the same characterization of local nilpotence survives in noncommutative deformations of the affine plane. In this direction, we consider the differential Ore extensions
\[
A_h=\Bbbk[x][t;h(x)\partial_x]
=\Bbbk\langle x,t\mid tx-xt=h(x)\rangle,
\]
where $h\in\Bbbk[x]\setminus\Bbbk$. Our main result shows that Pan's criterion extends completely to this family: for every $D\in\der(A_h)$. More precisely, our main result,
Theorem \ref{main_Ah}, shows that a derivation of $A_h$ is locally nilpotent if and only if its isotropy group contains automorphisms of arbitrarily large degree.

The two extremal cases of the family $A_h$ recover two classical algebras. If $h=0$, then $A_h=\Bbbk[x,t]$, and Pan's theorem applies. If $h\in\Bbbk^\ast$, then, after a nonzero scalar rescaling, $A_h$ is isomorphic to the first Weyl algebra $A_1$. If $h\in\Bbbk^\ast$, then, after a nonzero
scalar rescaling, one obtains the first Weyl algebra
\[
A_1=\Bbbk\langle p,q\mid pq-qp=1\rangle.
\]

The first Weyl algebra requires a separate analysis. It is not known whether, for an arbitrary nonzero derivation $D$, the existence of automorphisms of arbitrarily large degree in $\Aut_D(A_1)$ necessarily implies that $D$ is locally nilpotent. However, using Dixmier's results
on the first Weyl algebra together with explicit computations of isotropy groups, we prove in Theorem \ref{Pan_A1_LF} that every nonzero locally finite derivation $D$ of $A_1$ is locally nilpotent if and only
if its isotropy group $\Aut_D(A_1)$ contains automorphisms of arbitrarily large degree.

These results fit naturally into the classification of Ore extensions of $\Bbbk[x]$. By results of Awami, Van den Bergh and Van Oystaeyen \cite[Section 2.1]{AwamiVanDenBerghVanOystaeyen} and Alev and Dumas
\cite[Proposition 3.2]{AlevDumas}, every Ore extension
\[
S=\Bbbk[x][y;\sigma,\delta],
\]
with $\sigma\in\Aut(\Bbbk[x])$ and $\delta$ a $\sigma$-derivation, is isomorphic to a polynomial algebra, a quantum plane, a first quantum Weyl algebra, or a differential Ore extension. The remaining quantum cases are discussed in Section \ref{section_quantum}. By collecting well-known results from the literature concerning their automorphism groups and locally nilpotent derivations, we make explicit how the same Pan-type
equivalence behaves for the quantum plane and the first quantum Weyl algebra.

Beyond the class of Ore extensions, we also consider the free associative algebra $F=\Bbbk\langle x,y\rangle$, where the relation between local nilpotence
and the isotropy group exhibits a different behavior. The implication from local nilpotence to unbounded degree
of the isotropy group still holds, as a consequence of the
triangularization of locally nilpotent derivations of $F$. The converse, however, fails: we exhibit an inner derivation which is not locally nilpotent but whose isotropy group contains automorphisms of arbitrarily
large degree.

Abelianization provides a conceptual explanation for this failure. By the classical results of Czerniakiewicz and Makar-Limanov (see \cite{Czerniakiewicz1971,MakarLimanov}), automorphisms of $F$ are completely detected by their images in $\Bbbk[x,y]$. Derivations behave differently: the induced map $\pi_\ast:\der(F)\rightarrow\der(\Bbbk[x,y])$ has a nontrivial kernel, so this map is not injective. We show in Corollary \ref{kernel_lnd_free} that its kernel contains no nonzero locally nilpotent derivations:
\[
\ker(\pi_\ast)\cap\LND(F)=\{0\}.
\]

Although automorphisms of $F$ are completely determined by their abelianizations, local nilpotence of the induced derivation on the commutative quotient does not imply local nilpotence of the original derivation. Our inner counterexample in Proposition \ref{free_counterexample}
has zero abelianization. Furthermore, Proposition \ref{prop:free-nonzero-abelianization} shows
that the converse to Pan's criterion can fail even when the induced derivation is nonzero and locally nilpotent.

\section{The first Weyl algebra}\label{section_A1}

Let $A_1=\Bbbk\langle p,q\mid pq-qp=1\rangle$ be the first Weyl algebra: its center is $\operatorname{Z}(A_1)=\Bbbk$ \cite[p. 210]{Dixmier}. Every element of $A_1$ has a unique PBW expression $u=\sum_{i,j\geq0}c_{ij}p^iq^j$ and we define $
\deg(u)=\max\{i+j\mid c_{ij}\neq0\}$, and, for $\rho\in\Aut(A_1)$,
\[
\deg(\rho)=\max\{\deg(\rho(p)),\deg(\rho(q))\}.
\]

It is well known that every derivation of $A_1$ is inner; see \cite[p. 210]{Dixmier}. Then, every $D\in\der(A_1)$ has the form $D=\ad_w:=[w,\_]$, for some $w\in A_1$.

The following simple consequence will be useful throughout this section.

\begin{lemma}\label{lemma_inner_A1}
Let $w\in A_1$. Then
\[
\Aut_{\ad_w}(A_1)
=
\{\rho\in\Aut(A_1)\mid \rho(w)-w\in\Bbbk\}.
\]
\end{lemma}

\begin{proof}
For every $\rho\in\Aut(A_1)$, $\rho\ad_w\rho^{-1}=\ad_{\rho(w)}$. Then, $\rho\in\Aut_{\ad_w}(A_1)$ if and only if $\ad_{\rho(w)}=\ad_w$, which is equivalent to $\rho(w)-w\in  \Z(A_1)$. Since $ \Z(A_1)=\Bbbk$, the result follows.
\end{proof}

We recall the results of Dixmier \cite{Dixmier} that are relevant to locally finite derivations. For $w\in A_1$, let
\[
F_w:=
\left\{
u\in A_1
\ \middle|\
\dim_{\Bbbk}
\Span_{\Bbbk}\{u,\ad_w(u),\ad_w^2(u),\ldots\}<\infty
\right\},
\]
and
\[
N_w:=
\left\{
u\in A_1
\ \middle|\
\ad_w^n(u)=0
\text{ for some }n\geq1
\right\}.
\]

For $\lambda\in\Bbbk$, let
\[
E_w(\lambda):
=
\{u\in A_1\mid \ad_w(u)=\lambda u\},
\]
and denote
\[
E_w=\bigoplus_{\lambda\in\Bbbk}E_w(\lambda).
\]

Dixmier \cite[Corollary 6.6]{Dixmier} proves that
\[
F_w=N_w
\quad\text{or}\quad
F_w=E_w.
\] 

Therefore, if $\ad_w$ is locally finite, then $F_w=A_1$ and there are only two possibilities: either $N_w=A_1$, in which case $\ad_w$ is locally nilpotent, or $E_w=A_1$, in which case $\ad_w$ is locally finite but not locally nilpotent. Thus, we obtain the following statement (see \cite[Corollary 6.6 and Theorems 9.1, 9.2]{Dixmier} and the proof of \cite[Corollary 9.3]{Dixmier}).

\begin{proposition}[{\cite[Corollary 6.6, Theorems 9.1 - 9.2]{Dixmier}}]\label{prop_Dixmier_LF}
Let $0\neq D=\ad_w\in\LFDer(A_1)$. Then, exactly one of the following occurs:
\begin{itemize}
    \item[i)] $D$ is locally nilpotent, and there exists
    $\varphi\in\Aut(A_1)$ such that
    $\varphi(w)=f(p)$, for some $f(p)\in\Bbbk[p]\setminus\Bbbk$;

    \item[ii)] $D$ is locally finite but not locally nilpotent, and
    there exist $\varphi\in\Aut(A_1)$,
    $\lambda\in\Bbbk^\ast$ and $\mu\in\Bbbk$ such that $\varphi(w)=\lambda pq+\mu$.
\end{itemize}
\end{proposition}

\begin{proof}
Since $D=\ad_w$ is locally finite, $F_w=A_1$. By
\cite[Cor. 6.6]{Dixmier}, $F_w=N_w$
or $F_w=E_w$.

If $F_w=N_w=A_1$, then $\ad_w$ is locally nilpotent. By
\cite[Thm. 9.1]{Dixmier}, there exists $\varphi\in\Aut(A_1)$ such that $\varphi(w)\in\Bbbk[p]$.

If $F_w=E_w=A_1$, then $D$ is locally finite and, since $D\neq0$, it is not locally nilpotent. By \cite[Thm. 9.2]{Dixmier}, together with the reduction used in the proof of \cite[Cor. 9.3]{Dixmier} over an algebraically closed field, one may assume $\varphi(w)=\lambda pq+\mu$, for some $\lambda\in\Bbbk^\ast$ and $\mu\in\Bbbk$.
\end{proof}

Let $f(p)\in\Bbbk[p]\setminus\Bbbk$ and consider $D=\ad_{f(p)} \in \der(A_1)$. The following lemma determines the centralizer of $f(p)$ in $A_1$.

\begin{lemma}\label{centralizer_fp}
If $f(p)\in\Bbbk[p]\setminus\Bbbk$, then $C_{A_1}(f(p))=\Bbbk[p]$.
\end{lemma}

\begin{proof}
The inclusion $\Bbbk[p]\subseteq C_{A_1}(f(p))$ is clear. Conversely, let
\[
u=\sum_{j=0}^{m}u_j(p)q^j,
\quad
u_m(p)\neq0.
\]
Suppose that $m\geq1$. Since $[p,q]=1$, the coefficient of
$q^{m-1}$ in $[f(p),u]$ is $m u_m(p)f'(p)$. Since $\operatorname{char}(\Bbbk)=0$ and $f$ is nonconstant,
$f'(p)\neq0$. Then, $[f(p),u]\neq0$, a contradiction. Thus, $m=0$ and $u\in\Bbbk[p]$.
\end{proof}

\begin{proposition}\label{isotropy_fp}
Let $D=\ad_{f(p)}$, where $f(p)\in\Bbbk[p]\setminus\Bbbk$. Then
\[
\Aut_D(A_1)
=
\left\{
\rho_{a,b,g}\in\Aut(A_1)
\ \middle|\
f(ap+b)-f(p)\in\Bbbk
\right\},
\]
where $\rho_{a,b,g}(p)=ap+b$ and $\rho_{a,b,g}(q)=a^{-1}q+g(p)$, with $a\in\Bbbk^\ast$, $b\in\Bbbk$ and $g(p)\in\Bbbk[p]$. In particular, $\Aut_D(A_1)$ contains automorphisms of arbitrarily large degree.
\end{proposition}

\begin{proof}
Let $\rho\in\Aut_D(A_1)$. By Lemma \ref{lemma_inner_A1},
$\rho(f(p))-f(p)\in\Bbbk$. Then, $\rho(f(p))=f(p)+c$,
for some $c\in\Bbbk$. Since adding a scalar does not change the
centralizer,
\[
\rho(C_{A_1}(f(p)))
=
C_{A_1}(\rho(f(p)))
=
C_{A_1}(f(p)).
\]

By Lemma \ref{centralizer_fp}, $\rho(\Bbbk[p])=\Bbbk[p]$. Thus, $\rho$ restricts to an automorphism of the polynomial algebra $\Bbbk[p]$, and consequently $\rho(p)=ap+b$, for some $a\in\Bbbk^\ast$ and $b\in\Bbbk$. Using the Weyl relation and $\rho(p)=ap+b$, we obtain
\[
1=[\rho(p),\rho(q)]
=[ap+b,\rho(q)]
=a[p,\rho(q)].
\]

Then, $[p,\rho(q)]=a^{-1}=[p,a^{-1}q]$ and, therefore, $[p,\rho(q)-a^{-1}q]=0$. Since $C_{A_1}(p)=\Bbbk[p]$, there exists $g(p)\in\Bbbk[p]$ such that $\rho(q)=a^{-1}q+g(p)$. Finally, the condition of Lemma \ref{lemma_inner_A1} becomes $f(ap+b)-f(p)\in\Bbbk$.

Conversely, every map of this form is an automorphism of $A_1$ and satisfies $\rho(f(p))-f(p)\in\Bbbk$, then belongs to $\Aut_D(A_1)$. Additionally, let $a=1$ and $b=0$, we obtain automorphisms $p\mapsto p$, $q\mapsto q+g(p)$, with $g(p)\in\Bbbk[p]$ arbitrary: their degrees are unbounded.
\end{proof}

Let $H:=pq$. Then, $[H,p]=-p$ and $[H,q]=q$. More generally,
\[
[H,p^iq^j]=(j-i)p^iq^j.
\]

Thus, every PBW monomial is an eigenvector of $\ad_H$, and then $\ad_H$ is locally finite.

\begin{lemma}\label{eigenspaces_H}
Let $H=pq$. Then,
\[
E_H(-1)=p\Bbbk[H]
\quad\text{and}\quad
E_H(1)=\Bbbk[H]q.
\]
\end{lemma}

\begin{proof}
Since $[H,p^iq^j]=(j-i)p^iq^j$, the eigenspace associated to $-1$ is spanned by the monomials
$p^{n+1}q^n$, $n\geq0$. Using $pH=(H+1)p$, we obtain inductively,
\[
p^nq^n=H(H+1)\cdots(H+n-1).
\]

Therefore, $p^{n+1}q^n\in p\Bbbk[H]$. The reverse inclusion follows from $[H,p]=-p$ and $[H,H]=0$. Thus, $E_H(-1)=p\Bbbk[H]$. The second claim is analogous.
\end{proof}

\begin{proposition}\label{isotropy_H}
Let $H=pq$. Then,
\[
\Aut_{\ad_H}(A_1)
=
\{\rho_a\mid
\rho_a(p)=ap,\;
\rho_a(q)=a^{-1}q,\;
a\in\Bbbk^\ast\}.
\]

In particular, $\Aut_{\ad_H}(A_1)\simeq\Bbbk^\ast$ and every automorphism in this isotropy group has degree one.
\end{proposition}

\begin{proof}
Let $\rho\in\Aut_{\ad_H}(A_1)$. Since $\ad_H(\rho(p))
=\rho(\ad_H(p))=-\rho(p)$, Lemma \ref{eigenspaces_H} gives
$\rho(p)=p f(H)$, for some $f(H)\in\Bbbk[H]$. Similarly, $\rho(q)=g(H)q$, for some $g(H)\in\Bbbk[H]$.

On the other hand, Lemma \ref{lemma_inner_A1} gives
$\rho(H)=H+c$, for some $c\in\Bbbk$. Let $S(H)=f(H)g(H)$. Using $pS(H)=S(H+1)p$, we obtain
\[
\rho(H)
=\rho(p)\rho(q)=pS(H)q=S(H+1)H.
\]

Then, $S(H+1)H=H+c$. Since the left-hand side belongs to $H\Bbbk[H]$, its constant term is zero: $c=0$.
Consequently, $S(H+1)=1$. It follows that $f(H)g(H)=1$. Since the only units of $\Bbbk[H]$
are the nonzero scalars, there exists $a\in\Bbbk^\ast$ such that $f(H)=a$ and $g(H)=a^{-1}$. Thus,
\[
\rho(p)=ap,
\quad
\rho(q)=a^{-1}q.
\]

Conversely, every automorphism of this form fixes $H$ and, therefore, commutes with $\ad_H$.
\end{proof}

\begin{theorem}\label{Pan_A1_LF}
Let $0\neq D\in\LFDer(A_1)$. Then, the following conditions are equivalent:
\begin{itemize}
    \item[(i)] $D$ is locally nilpotent;
    \item[(ii)] $\Aut_D(A_1)$ contains automorphisms of arbitrarily
    large degree.
\end{itemize}
\end{theorem}

\begin{proof}
Let $D=\ad_w$, with $w \in A_1$. Suppose first that $D$ is locally nilpotent. By Proposition \ref{prop_Dixmier_LF}, after conjugation by a fixed automorphism of $A_1$, we may assume $w=f(p)$,
for some $f(p)\in\Bbbk[p]\setminus\Bbbk$. By Proposition \ref{isotropy_fp}, the isotropy contains all automorphisms
\[
p\longmapsto p,
\quad
q\longmapsto q+g(p),
\]
and consequently has unbounded degree.

Conversely, suppose that $D$ is locally finite but not locally nilpotent. By Proposition \ref{prop_Dixmier_LF}, after conjugation and multiplication by a nonzero scalar, we may assume $D=\ad_H$, where $H=pq$. By Proposition \ref{isotropy_H},
\[
\Aut_D(A_1)
=
\{p\mapsto ap,\ q\mapsto a^{-1}q\mid a\in\Bbbk^\ast\},
\]
and hence its degree is bounded by one. Since conjugation by a fixed automorphism preserves the distinction between bounded and unbounded degree, the result follows.
\end{proof}

\begin{problem}\label{problem_A1}
Let $w\in A_1\setminus\Bbbk$. If
$\{\rho\in\Aut(A_1)\mid \rho(w)-w\in\Bbbk\}$ has unbounded degree, must $\ad_w$ be locally nilpotent? Theorem \ref{Pan_A1_LF} gives an affirmative answer whenever
$\ad_w$ is locally finite.
\end{problem}

\section{Differential Ore extensions $A_h$}\label{section_Ah}

Throughout this section, we consider $A_h=\Bbbk[x][t;h(x)\partial_x]$, with $h\in\Bbbk[x]\setminus\Bbbk$. Equivalently, $A_h$ is generated by $x$ and $t$ subject to $tx=xt+h(x)$. We write $N=\deg(h)\geq1$. For $g(x)\in\Bbbk[x]$, let $D_{g(x)}$ be the derivation of $A_h$ defined by
\[
D_{g(x)}(x)=0,
\quad
D_{g(x)}(t)=g(x).
\]

We recall the following result:

\begin{proposition}[{\cite[Prop. 2]{ISF2021}}]\label{prop_LND_Ah}
Let $h\in\Bbbk[x]\setminus\Bbbk$. Then
\[
\LND(A_h)
=
\{D_{g(x)}\mid g(x)\in\Bbbk[x]\}.
\]
Moreover, $\ML(A_h)=\Bbbk[x]$.
\end{proposition}

The automorphism group was described by G. Benkart, S. A. Lopes and M. Ondrus \cite{BLO2015}.

\begin{theorem}[{\cite[Thm. 8.3]{BLO2015}}]\label{Aut_Ah_general} Let $\deg(h)=N\geq1$ and
\[
\mathbb P_h
=
\{(a,b)\in\Bbbk^\ast\times\Bbbk
\mid
h(ax+b)=a^Nh(x)\}.
\]
Then every $\rho\in\Aut(A_h)$ can be written as $\rho=\sigma_{r(x)}\circ\tau_{a,b}$, where $(a,b)\in\mathbb P_h$, $r(x)\in\Bbbk[x]$, 
\[
\sigma_{r(x)}(x)=x,
\quad
\sigma_{r(x)}(t)=t+r(x),
\]
and
\[
\tau_{a,b}(x)=ax+b,
\quad
\tau_{a,b}(t)=a^{N-1}t.
\]
\end{theorem}

We use the following normalization obtained in
\cite{BaltazarSilvaMartini}.

\begin{proposition}[{\cite[Prop. 13]{BaltazarSilvaMartini}}]\label{normalization_Ah}
Let $h(x)\in\Bbbk[x]$, with $\deg(h)=N\geq1$. There exists a monic
polynomial $h^\ast\in\Bbbk[x]$ such that $A_h\simeq A_{h^\ast}$,
the coefficient of degree $N-1$ of $h^\ast$ is zero, and
\[
\Aut(A_{h^\ast})
=
\{\sigma_{r(x)}\circ\tau_a
\mid
h^\ast(ax)=a^Nh^\ast(x)\},
\]
where $\tau_a=\tau_{a,0}$.
\end{proposition}

From now on, we replace $A_h$ by this isomorphic normalized and write again $h$ for $h^\ast$. Thus, $h$ is monic and its coefficient of degree $N-1$ is zero. Every automorphism can be written as $\rho=\sigma_{r(x)}\circ\tau_a$, where
\[
\rho(x)=ax,
\quad
\rho(t)=a^{N-1}(t+r(x)),
\]
and
\[
h(ax)=a^Nh(x).
\]

Since this normalization is obtained by an affine change of the base variable together with a nonzero scalar rescaling, both local nilpotence and the distinction between bounded and unbounded degree are preserved.

\begin{proposition}\label{LND_large_Ah}
If $D\in\LND(A_h)$, then $\Aut_D(A_h)$ contains automorphisms of arbitrarily large degree.
\end{proposition}

\begin{proof}
By Proposition \ref{prop_LND_Ah}, there exists $g(x)\in\Bbbk[x]$ such that $D(x)=0$ and $D(t)=g(x)$. For every $r(x)\in\Bbbk[x]$, consider
\[
\sigma_{r(x)}(x)=x,
\quad
\sigma_{r(x)}(t)=t+r(x).
\]

Then, $D(\sigma_{r(x)}(x))=0=\sigma_{r(x)}(D(x))$,
and
\[
D(\sigma_{r(x)}(t))
=
D(t+r(x))
=
g(x)
=
\sigma_{r(x)}(D(t)).
\]

Therefore, $\sigma_{r(x)}\in\Aut_D(A_h)$, for every $r(x)\in\Bbbk[x]$.
\end{proof}

We recall some well-known elementary relations in the algebra $A_h$ that will be used throughout this section. Proofs can be found, for instance, in \cite{BaltazarSilvaMartini}. Every element $u\in A_h$ can be written uniquely as $u=\sum_{i=0}^{m}u_i(x)t^i$, where $u_i(x)\in\Bbbk[x]$. If $u_m(x)\neq0$, we write $\degt(u)=m$, and we denote $\degt(0)=-\infty$.

\begin{lemma}\label{tf_Ah}
For every $f(x)\in\Bbbk[x]$, we have $tf(x)=f(x)t+f'(x)h(x)$.
\end{lemma}

\begin{corollary}\label{tg_Ah}
Let $g=\sum_{i=0}^{n}g_i(x)t^i\in A_h$. Then, $tg=gt+h(x)\widetilde g$, where $\widetilde g
=\sum_{i=0}^{n}g_i'(x)t^i$ and $\degt(\widetilde g)\leq\degt(g)$.
\end{corollary}

\begin{lemma}\label{ti_x_Ah}
For every $i\geq1$, there exists $u_i\in A_h$ such that
\[
[t^i,x]
=
i h(x)t^{i-1}+h(x)u_i
\]
and $\degt(u_i)<i-1$.
\end{lemma}

\begin{lemma}\label{t_plus_r_Ah}
Let $r(x)\in\Bbbk[x]$. For every $i\geq1$, there exists $g_i\in A_h$
such that
\[
(t+r(x))^i
=
t^i+i r(x)t^{i-1}+g_i
\]
and $\degt(g_i)\leq i-2$.
\end{lemma}

We also use the following elementary description of the centralizer of $x$ in $A_h$.

\begin{lemma}\label{centralizer_x_Ah}
Let $h(x)\neq0$. Then, $C_{A_h}(x)=\Bbbk[x]$. In particular, $ \Z(A_h)=\Bbbk$.
\end{lemma}

For $u=\sum_{i,j}c_{ij}x^it^j\in A_h$, we consider the total degree $\deg(u)=\max\{i+j\mid c_{ij}\neq0\}$ with the convention $\degt(0)=\deg(0)=-\infty$. For $\rho\in\Aut(A_h)$, we also define $\deg(\rho)
=\max\{\deg(\rho(x)),\deg(\rho(t))\}$. If $\rho=\sigma_{r(x)}\circ\tau_a$, then
\[
\rho(x)=ax,
\quad
\rho(t)=a^{N-1}(t+r(x)),
\]
and, consequently, $\deg(\rho)=\max\{1,\deg(r(x))\}$. Thus, we immediately obtain the following observation.

\begin{lemma}\label{large_degree_r}
Let $G\subseteq\Aut(A_h)$. Then, $G$ contains automorphisms of arbitrarily large degree if and only if there is a sequence
\[
\rho_n
=
\sigma_{r_n(x)}\circ\tau_{a_n}\in G
\]
such that $\deg(r_n)\longrightarrow\infty$.
\end{lemma}

The following two results will play a central role in the proof of the main theorem of this section.

\begin{lemma}\label{Dx_positive_t}
Let $D\in\der(A_h)$ and suppose that $\degt(D(x))\geq1$. Then, there exists $M_D\geq0$ such that, for every
$\rho=\sigma_{r(x)}\circ\tau_a\in\Aut_D(A_h)$, we have $\deg(r(x))\leq M_D$.
\end{lemma}

\begin{proof}
Write $D(x)=u=\sum_{i=0}^{m}u_i(x)t^i$, where $u_m(x)\neq0$ and $m\geq1$. Let $\rho=\sigma_{r(x)}\circ\tau_a\in\Aut_D(A_h)$. Since
$\rho(x)=ax$, the state $\rho D=D\rho$ gives
\[
\rho(D(x))
=
D(\rho(x))
=
D(ax)
=
aD(x)
=
au.
\]

On the other hand, since $\rho(t)=a^{N-1}(t+r(x))$, we have
\[
\rho(u)
=
\sum_{i=0}^{m}
u_i(ax)a^{i(N-1)}(t+r(x))^i.
\]

By Lemma \ref{t_plus_r_Ah}, the coefficient of $t^{m-1}$ in
$(t+r(x))^m$ is $mr(x)$. Then, notice that the coefficient of $t^{m-1}$ in
$\rho(u)$ is
\[
m a^{m(N-1)}u_m(ax)r(x)
+
a^{(m-1)(N-1)}u_{m-1}(ax).
\]
Comparing this with the coefficient of $t^{m-1}$ in $au$, we obtain
\[
m a^{m(N-1)}u_m(ax)r(x)
=
a u_{m-1}(x)
-
a^{(m-1)(N-1)}u_{m-1}(ax).
\tag{*}
\]

If $u_{m-1}=0$, then the right-hand side of $(*)$ is zero. Since $m\neq0$, $a\neq0$, and $u_m(ax)\neq0$, it follows that $r(x)=0$.

Suppose that $u_{m-1}\neq0$ and $r(x)\neq0$. Since
$\deg(u_m(ax))=\deg(u_m)$, the left-hand side of $(*)$ has degree $\deg(u_m)+\deg(r)$. On the other hand, each term on the right-hand side of $(*)$ has
degree $\deg(u_{m-1})$, and therefore
\[
\deg (
a u_{m-1}(x)
-
a^{(m-1)(N-1)}u_{m-1}(ax)
)
\leq
\deg(u_{m-1}).
\]

Consequently, $\deg(r)\leq\deg(u_{m-1})-\deg(u_m)$. Therefore, the degree of $r(x)$ is bounded independently of $\rho$. In fact, we can choose
\[
M_D=\max\bigl\{0,\deg(u_{m-1})-\deg(u_m)\bigr\}.
\]
\end{proof}

\begin{lemma}\label{Dx_polynomial}
Let $D\in\der(A_h)$ and suppose that $D(x)=p(x)\in\Bbbk[x]\setminus\{0\}$. Then, there exists $M_D\geq0$ such that, for every $\rho=\sigma_{r(x)}\circ\tau_a\in\Aut_D(A_h)$, we have
$\deg(r(x))\leq M_D$.
\end{lemma}

\begin{proof}
Write $D(t)=v\in A_h$. Applying $D$ to the relation $[t,x]=h(x)$ gives $[D(t),x]+[t,D(x)]=D(h(x))$. Since $D(x)=p(x)\in\Bbbk[x]$, we have also $D(h(x))=h'(x)p(x)$. Moreover, by Lemma \ref{tf_Ah}, $[t,p(x)]=h(x)p'(x)$. Therefore,
\[
[v,x]
=
h'(x)p(x)-h(x)p'(x)
\in\Bbbk[x].
\tag{1}
\]

We claim that $v$ has $t$-degree at most one. Suppose that
$v=\sum_{i=0}^{s}v_i(x)t^i$, with $v_s(x)\neq0$ and $s\geq2$. Since
\[
[v,x]
=
\sum_{i=0}^{s}v_i(x)[t^i,x],
\]
Lemma \ref{ti_x_Ah} shows that the coefficient of $t^{s-1}$ coming from the term $v_s(x)t^s$ is
$s v_s(x)h(x)$. For $i\leq s-1$, the commutator $[t^i,x]$ has $t$-degree at most $i-1\leq s-2$, so no other term contributes to the coefficient of $t^{s-1}$. Thus, this coefficient in $[v,x]$ is $s v_s(x)h(x)$. Since $\operatorname{char}(\Bbbk)=0$, $v_s(x)\neq0$ and $h(x)\neq0$, this coefficient is nonzero, contradicting $(1)$. Then, $D(t)=b(x)t+c(x)$, for some $b(x),c(x)\in\Bbbk[x]$. Substituting into
$(1)$, we obtain
\[
b(x)h(x)=h'(x)p(x)-h(x)p'(x).
\tag{2}
\]

Let $\rho=\sigma_{r(x)}\circ\tau_a\in\Aut_D(A_h)$. Since
$\rho(x)=ax$ and $\rho(t)=a^{N-1}(t+r(x))$, the condition $\rho D=D\rho$, applied to $t$, gives
\[
a^{N-1}b(ax)(t+r(x))+c(ax)=a^{N-1}
\bigl(b(x)t+c(x)+p(x)r'(x)\bigr),
\]

Comparing the coefficients of $t$, we obtain $b(ax)=b(x)$. Also, comparing the terms of $t$-degree zero, we obtain
\[
p(x)r'(x)-b(x)r(x)=a^{1-N}c(ax)-c(x).
\tag{3}
\]

Since $p(x)$, $b(x)$, and $c(x)$ depend only on $D$, they are fixed. Moreover, since $a\in\Bbbk^\ast$, $\deg(c(ax))=\deg(c)$, whenever $c\neq0$. Then, we have that $\deg\bigl(a^{1-N}c(ax)-c(x)\bigr)\leq\deg(c)$, so the degree of the right-hand side of $(3)$ is bounded
independently of $\rho$. We show that there exists a bound for $\deg(r(x))$ depending only on $D$.

Suppose that $b(x)=0$. Constant polynomials $r(x)$ already have bounded degree, so assume that $r(x)$ is nonconstant. Denote $d=\deg(r)$ and $n=\deg(p)$. Thus, $\deg\bigl(p(x)r'(x)\bigr)=n+d-1$. If $c\neq0$, the right-hand side of $(3)$ has degree at most
$\deg(c)$. Then, $n+d-1\leq\deg(c)$, and so $d\leq\deg(c)-n+1$. If $c=0$, then $(3)$ gives $p(x)r'(x)=0$, and then $r'(x)=0$, contradicting the assumption that $r(x)$ is nonconstant. Thus, in either case, the degree of $r(x)$ is bounded independently of $\rho$.

Suppose that $b(x)\neq0$, as in the previous case, $n=\deg(p)$, $s=\deg(b)$, $d=\deg(r)$, with $r(x)$ nonconstant. Then, $\deg\bigl(p(x)r'(x)\bigr)=n+d-1$ and $\deg\bigl(b(x)r(x)\bigr)=s+d$. If $s\neq n-1$, these two degrees are different. Then, no cancellation of their leading terms is possible, and $\deg\bigl(p(x)r'(x)-b(x)r(x)\bigr)=d+\max\{n-1,s\}$. Since the right-hand side of $(3)$ has bounded degree, $d$ is bounded. It remains to consider the case $s=n-1$. Let $p_n$, $b_s$, and $r_d$ denote the leading coefficients of $p(x)$, $b(x)$, and $r(x)$, respectively. The coefficient of
$x^{n+d-1}$ in $p(x)r'(x)-b(x)r(x)$ is $r_d(dp_n-b_s)$. Thus, the leading term can cancel only if $d=\frac{b_s}{p_n}$. If this does not hold, then
\[
\deg\bigl(p(x)r'(x)-b(x)r(x)\bigr)=n+d-1,
\]
and again the bound on the degree of the right-hand side of
$(3)$ gives a bound on $d$. If $d=\frac{b_s}{p_n}$, then $d$ is fixed, since $p_n$ and $b_s$ depend only on $D$.

Therefore, in every case, the degree of $r(x)$ is bounded by a constant depending only on $D$. Therefore, there exists $M_D\geq0$ such that
\[
\deg(r(x))\leq M_D,
\]
for every $\rho\in\Aut_D(A_h)$.
\end{proof}

\begin{theorem}\label{main_Ah}
Let $h\in\Bbbk[x]\setminus\Bbbk$. For every $D\in\der(A_h)$,
the following conditions are equivalent:
\begin{itemize}
    \item[(i)] $D$ is locally nilpotent;
    \item[(ii)] $\Aut_D(A_h)$ contains automorphisms of arbitrarily
    large degree.
\end{itemize}
\end{theorem}

\begin{proof}
The implication $(i)\Rightarrow(ii)$ follows from
Proposition \ref{LND_large_Ah}.

Conversely, suppose that $\Aut_D(A_h)$ contains automorphisms of arbitrarily large degree. By Lemma \ref{large_degree_r}, there exists
a sequence $\rho_n=\sigma_{r_n(x)}\circ\tau_{a_n}
\in\Aut_D(A_h)$ such that $\deg(r_n)\longrightarrow\infty$.

We claim that $D(x)=0$. Indeed, if $D(x)$ contains a positive power of $t$, then $\degt(D(x))\geq1$, and Lemma \ref{Dx_positive_t} implies that the degrees of the
polynomials $r_n(x)$ are bounded, a contradiction. Then, $D(x)\in\Bbbk[x]$. If $D(x)\neq0$, then Lemma \ref{Dx_polynomial} again implies that the degrees of the polynomials $r_n(x)$ are bounded, which is impossible. Therefore, $D(x)=0$.

Applying $D$ to the defining relation $[t,x]=h(x)$, we obtain
\[
[D(t),x]=D(h(x))=h'(x)D(x)=0.
\]
By Lemma \ref{centralizer_x_Ah}, $C_{A_h}(x)=\Bbbk[x]$, and therefore $D(t)\in\Bbbk[x]$. Thus,
\[
D(x)=0,
\quad
D(t)=g(x),
\]
for some $g(x)\in\Bbbk[x]$. Then, $D=D_{g(x)}$, and Proposition \ref{prop_LND_Ah} gives $D\in\LND(A_h)$.
\end{proof}

\begin{remark}[Comparison with Pan's argument]

Pan \cite{Pan2022} uses in the proof of Theorem~\ref{thm_Pan} the invariant $e(D)$, already considered in \cite{BaltazarPan2021}, which
counts the $D$-stable reduced principal ideals of height one in $\Bbbk[x,y]$. This invariant organizes the proof according to the invariant curves of $D$. The case $0<e(D)<\infty$ in Pan's proof is particularly similar to our argument: after suitable reductions, unbounded isotropy is encoded by a polynomial parameter $P(x)$ of arbitrarily large degree, and the commutation equations rule out a nonzero $\partial_x$-component of $D$. For $A_h$, the explicit description of the whole automorphism group makes such a case distinction unnecessary. After normalization,
\[
\Aut(A_h)
=
\{
\sigma_{r(x)}\circ\tau_a
\ \mid\
h(ax)=a^Nh(x)
\},
\]
and unbounded degree is exactly reflected in the polynomial parameter $r(x)$. The two boundedness lemmas above show that polynomials $r(x)$ of arbitrarily large degree cannot occur when $D(x)\neq0$. Then,
unbounded isotropy forces $D(x)=0$, and the defining relation then gives $D(t)\in\Bbbk[x]$.

\end{remark}

\begin{remark}[The degree filtration]
The description of $\Aut(A_h)$ provides a natural analogue of the degree filtration considered in the polynomial case. For $d\geq0$, denote
\[
\mathcal A_d(A_h)=\left\{
\sigma_{r(x)}\circ\tau_a\in\Aut(A_h)
\ \middle|\
\deg(r(x))\leq d
\right\}.
\]

Then,
\[
\Aut(A_h)=
\bigcup_{d\geq0}\mathcal A_d(A_h).
\]
Moreover, since $\deg(\sigma_{r(x)}\circ\tau_a)=\max\{1,\deg(r(x))\}$, for every $d\geq1$ we have
\[
\mathcal A_d(A_h)=\{\rho\in\Aut(A_h)\mid \deg(\rho)\leq d\}.
\]

Notice that $\mathcal A_0(A_h)$ consists of the automorphisms for which $r(x)$ is constant, and these automorphisms have degree one. 

Let $G_h=\{a\in\Bbbk^\ast\mid h(ax)=a^Nh(x)\}$. Writing $r(x)=r_0+r_1x+\cdots+r_dx^d$, the set $\mathcal A_d(A_h)$ is naturally parametrized by $\Bbbk^{d+1}\times G_h$. For a fixed $D\in\der(A_h)$, the conditions $\rho D(x)=D\rho(x)$ and $\rho D(t)=D\rho(t)$, after writing these identities in PBW form, their coefficients give polynomial equations in $r_0,\ldots,r_d$ and $a$. Then,
$\Aut_D(A_h)\cap\mathcal A_d(A_h)$ is Zariski closed in this finite-dimensional parameter space. Furthermore, Theorem \ref{main_Ah} is equivalent to saying that $D$ is locally nilpotent if and only if $\Aut_D(A_h)$ is not contained in any $\mathcal A_d(A_h)$.
\end{remark}

\section{The free associative algebra}\label{section_free}

Let $F=\Bbbk\langle x,y\rangle$ be the free associative algebra on two generators. The behavior of locally nilpotent derivations of $F$ is closely related
to the corresponding theory for $\Bbbk[x,y]$. In fact, their triangulability can be obtained from Rentschler's theorem together with the classical results of Czerniakiewicz and Makar-Limanov on automorphisms of the free associative algebra of rank two.

We recall the noncommutative analogue of Rentschler's theorem. The following result is due to Drensky and Limanov \cite{DrenskyMakarLimanov}; see also
\cite{AlimbaevNaurazbekovaKozybaev,CrodeShestakov}.

\begin{proposition}[{\cite[Proposition 3.2]{DrenskyMakarLimanov}}]\label{free_triangularization}
Let $D\in\LND(F)$. Then, there exists a free generating system $x',y'$ of $F$ such that
\[
D(y')=0,
\quad
D(x')=f(y')
\]
for some $f(y')\in\Bbbk[y']$.
\end{proposition}

\begin{proposition}\label{free_LND_large}
If $0\neq D\in\LND(F)$, then $\Aut_D(F)$ contains automorphisms of
arbitrarily large degree.
\end{proposition}

\begin{proof}
By Proposition \ref{free_triangularization}, choose free generators $x',y'$ such that
\[
D(y')=0,
\quad
D(x')=f(y').
\]
For every $g(y')\in\Bbbk[y']$, define
$\sigma_g(x')=x'+g(y')$ and $\sigma_g(y')=y'$. Since $x',y'$ freely generate $F$, the map $\sigma_g$ is an
automorphism, with inverse $x'\mapsto x'-g(y')$ and $y'\longmapsto y'$. Moreover,
\[
D(\sigma_g(x'))=D(x'+g(y'))=f(y')=\sigma_g(D(x')),
\]
and the condition on $y'$ is immediate. Then, $\sigma_g\in\Aut_D(F)$. Taking $g$ of arbitrarily large degree gives automorphisms of arbitrarily large degree in the coordinates $x',y'$. Since the change of free generators is fixed, boundedness or unboundedness of degree is preserved under conjugation. Therefore, $\Aut_D(F)$ has unbounded degree.
\end{proof}

In contrast with Theorem \ref{thm_Pan}, the converse implication fails in the free associative algebra.

\begin{proposition}\label{free_counterexample}
Let $F=\Bbbk\langle x,y\rangle$, and define $D\in\der(F)$ by
\[
D(y)=0,
\quad
D(x)=[x,y].
\]

Then, $D$ is not locally nilpotent, while $\Aut_D(F)$ contains automorphisms of arbitrarily large degree.
\end{proposition}

\begin{proof}
We have $D=-\ad_y$. We claim that, for every $n\geq1$,
\[
D^n(x)=\sum_{i=0}^{n}
(-1)^i
\binom{n}{i}
y^ixy^{n-i}.
\]
The formula is clear for $n=1$. If it holds for some $n$, then, since $D(y)=0$,
\[
D\!\left(y^ixy^{n-i}\right)=y^ixy^{n-i+1}-y^{i+1}xy^{n-i},
\]
and the formula for $n+1$ follows from Pascal's identity.
The monomials $y^ixy^{n-i}$, with $0\leq i\leq n$, are pairwise distinct. Then, $D^n(x)\neq0$, for every $n\geq1$, and therefore $D$ is not locally nilpotent.

For every $g(y)\in\Bbbk[y]$, consider the automorphism
\[
\sigma_g(x)=x+g(y),
\quad
\sigma_g(y)=y.
\]

Since $\sigma_g(y)=y$, $\sigma_g\ad_y\sigma_g^{-1}=\ad_y$, and therefore $\sigma_gD=D\sigma_g$. Thus, $\sigma_g\in\Aut_D(F)$. Finally, $\deg(\sigma_g)=\max\{1,\deg(g)\}$, and then $\Aut_D(F)$ contains automorphisms of arbitrarily large degree.
\end{proof}

\subsection{Abelianization: automorphisms and derivations}

Let $\pi:F=\Bbbk\langle x,y\rangle\longrightarrow\Bbbk[x,y]$ be the natural abelianization map, and denote $I=\ker(\pi)$. Thus $I$ is the commutator ideal of $F$, which in rank two is the
two-sided ideal generated by $[x,y]$.

The behavior of automorphisms and derivations under abelianization is rather different. Every automorphism of $F$ preserves $I$ and therefore induces an automorphism of $\Bbbk[x,y]$. This gives a homomorphism
\[
\gamma:
\Aut(F)
\longrightarrow
\Aut(\Bbbk[x,y]).
\]

By \cite[Theorem 1]{Czerniakiewicz1971}, the homomorphism $\gamma$ is injective. On the other hand, the Jung--van der Kulk theorem \cite{Jung,vanDerKulk} implies that every automorphism of $\Bbbk[x,y]$ is tame. Since affine and triangular automorphisms lift naturally to $F$, the map $\gamma$ is also surjective. Consequently,
\[
\Aut(\Bbbk\langle x,y\rangle)
\simeq
\Aut(\Bbbk[x,y]).
\]

This is the classical result of Czerniakiewicz and Makar-Limanov; see also \cite{MakarLimanov}. Equivalently, every automorphism of $\Bbbk[x,y]$ has a unique lift to $F$. In particular, $\gamma(\rho)=\id$, then $\rho=\id$.

The situation for derivations is different. For every
$D\in\der(F)$,
\[
D([a,b])=[D(a),b]+[a,D(b)].
\]

Since $I$ is generated by commutators, it follows that
$D(I)\subseteq I$. Then, $D$ induces a derivation
$\pi_\ast(D)\in\der(\Bbbk[x,y])$. In this way, abelianization defines a linear map
\[
\pi_\ast:
\der(F)
\longrightarrow
\der(\Bbbk[x,y]).
\]

The map $\pi_\ast$ is not injective. Indeed, for every
$w\in F$, $\pi_\ast(\ad_w)=0$. Thus, a nonzero
derivation of $F$ may become trivial in the commutative quotient.

\begin{corollary}\label{kernel_lnd_free}
If $D\in\LND(F)$ and $\pi_\ast(D)=0$, then $D=0$. Equivalently,
\[
\ker(\pi_\ast)\cap\LND(F)=\{0\}.
\]
\end{corollary}

\begin{proof}
By Proposition \ref{free_triangularization}, there exist free generators $x',y'$ of $F$ such that
\[
D(y')=0,
\quad
D(x')=f(y'),
\]
for some $f(y')\in\Bbbk[y']$.

Let $\overline{x'}=\pi(x')$ and $\overline{y'}=\pi(y')$. Since $x',y'$ form a free generating system of $F$, there exists $\alpha\in\Aut(F)$ such that $\alpha(x)=x'$ and $\alpha(y)=y'$. Passing to the abelianization, $\alpha$ induces an automorphism of $\Bbbk[x,y]$ sending
\[
x\longmapsto\overline{x'},
\quad
y\longmapsto\overline{y'}.
\]

Then, $\overline{x'}$ and $\overline{y'}$ form a system of polynomial coordinates of $\Bbbk[x,y]$. In particular,
$\overline{y'}$ is transcendental over $\Bbbk$. Since $\pi_\ast(D)=0$, we have
\[
0
=
\pi_\ast(D)(\overline{x'})
=
\pi(D(x'))
=
\pi(f(y'))
=
f(\overline{y'}).
\]

As $\overline{y'}$ is transcendental over $\Bbbk$, it follows that $f=0$. Therefore, $D(x')=D(y')=0$. Since $x',y'$ generate $F$, we conclude that $D=0$.
\end{proof}

The preceding corollary shows that a nonzero locally nilpotent
derivation of $F$ cannot vanish under abelianization. However,
local nilpotence of the induced derivation $\pi_*(D)$ does not
imply local nilpotence of $D$. The following example shows that
the converse to Pan's criterion can fail even when $\pi_*(D)$
is nonzero and locally nilpotent.

\begin{proposition}\label{prop:free-nonzero-abelianization}
Let $F=\Bbbk\langle x,y\rangle$, and let
$D\in\operatorname{Der}(F)$ be defined by
\[
D(y)=0,
\quad
D(x)=1+[x,y].
\]
Then $D$ is not locally nilpotent, whereas $\pi_*(D)=\partial_x$ is nonzero and locally nilpotent. Moreover, $\operatorname{Aut}_D(F)$ contains automorphisms of arbitrarily large degree.
\end{proposition}

\begin{proof}
Notice that
\[
\pi_*(D)(x)=1,
\quad
\pi_*(D)(y)=0.
\]
Thus, $\pi_*(D)=\partial_x$, which is nonzero and locally nilpotent. For every $g(y)\in\Bbbk[y]$, consider the automorphism
\[
\sigma_g(x)=x+g(y),
\quad
\sigma_g(y)=y.
\]
Since $D(y)=0$, we have $D(g(y))=0$, and then $D(\sigma_g(x))=1+[x,y]$. On the other hand,
\[
\sigma_g(D(x))
 =1+[x+g(y),y]
 =1+[x,y].
\]

The condition $D(\sigma_g(y))=\sigma_g(D(y))$ is immediate.
Therefore, $\sigma_g\in\operatorname{Aut}_D(F)$. Taking $g(y)=y^m$, for $m\geq 1$, gives $\deg(\sigma_g)=m$, so the degrees are unbounded.

It remains to show that $D$ is not locally nilpotent.
For $u\in F$, the identity $D(y)=0$ gives $D([u,y])=[D(u),y]$. In particular, $D^2(x)=[[x,y],y]$. It follows inductively that, for every $n\geq 2$,
\[
D^n(x)=(-\operatorname{ad}_y)^n(x)
      =\sum_{i=0}^{n}(-1)^i\binom{n}{i}y^ixy^{n-i}.
\]

The elements $y^ixy^{n-i}$, with $0\leq i\leq n$, are pairwise
distinct, and the coefficient of $xy^n$ is $1$.
Consequently, $D^n(x)\neq 0$ for every $n\geq 2$.
Therefore, $D$ is not locally nilpotent.
\end{proof}

In this example, abelianization removes the commutator term
$[x,y]$ and retains the nonzero locally nilpotent derivation
$\partial_x$. Nevertheless, the original derivation is not
locally nilpotent and its isotropy group has unbounded degree.
Thus, Pan's criterion can fail for derivations of $F$ with nonzero abelianization.

\section{The quantum cases}\label{section_quantum}

M. Awami, M. Van den Bergh and F. Van Oystaeyen
\cite[Section 2.1]{AwamiVanDenBerghVanOystaeyen}, and J. Alev and F. Dumas \cite[Proposition 3.2]{AlevDumas}, proved that, given any $\Bbbk$-linear automorphism $\sigma$ of $\Bbbk[x]$ and $\delta$ a $\sigma$-derivation, the Ore extension $S=\Bbbk[x][y;\sigma,\delta]$ is isomorphic to one of the following algebras: $\Bbbk[x,y]$, a quantum plane, a quantum Weyl algebra, or a differential Ore extension. Since any automorphism of $\Bbbk[x]$ is such that $\sigma(x)=qx+b$, for some $q\in\Bbbk^\ast$ and $b\in\Bbbk$, following the proofs of \cite{AwamiVanDenBerghVanOystaeyen,AlevDumas}, we have:

\begin{itemize}

\item If $q\neq1$, then $S\simeq \Bbbk[x'][y;\sigma',\delta]$, where $x'=x+b(q-1)^{-1}$ and $\sigma'(x')=qx'$. Now, if $p(x)\in\Bbbk[x]$ and $r\in\Bbbk$ are such that
$\delta(x')=p(x')(1-q)x'+r$, then $(y-p(x'))x' =
qx'(y-p(x'))+r$. If $r=0$, it is easy to see that
$S\simeq \Bbbk_q[x',y']=\Bbbk\langle x',y'\mid y'x'=qx'y'\rangle$ for a suitable change of variables. If $r\neq0$, with $y''=r^{-1}(y-p(x'))$, we obtain
$S\simeq A_1^q(\Bbbk):=\Bbbk\langle x',y''\mid y''x'=qx'y''+1\rangle$.

\item If $q=1$ and $b=0$, then $S$ is either $\Bbbk[x,y]$ or a differential Ore extension, $S=\Bbbk[x][y;\delta]$.

\item If $q=1$ and $b\neq0$, then $S\simeq \Bbbk[x'][y;\sigma',\delta']$ where $x'=b^{-1}x$ $\sigma'(x')=x'+1$, and $\delta'(x')=b^{-1}\delta(x)$. Since
$(y+b^{-1}\delta(x))x'=(x'+1)(y+b^{-1}\delta(x))$, it follows that
$S\simeq \Bbbk[y'][x';-y'\partial_{y'}]$.

\end{itemize}

Let $A_q
=\Bbbk\langle x,y\mid xy=qyx\rangle$, where $q\in\Bbbk^\ast$ and $q\neq 1$. The automorphism computation of Alev and Chamarie \cite[Proposition 1.4.4]{AlevChamarie} gives, for $q\neq\pm1$,
\[
 \Aut(A_q)=\{x\mapsto\alpha x,\ y\mapsto\beta y\}
 \simeq(\Bbbk^*)^2.
\]

For $q=-1$ one has instead $\Aut(A_{-1})\simeq(\Bbbk^*)^2\rtimes \mathbb{Z}_2$, where the nontrivial element interchanges $x$ and $y$. Thus, for every $q\neq1$, all automorphisms of $A_q$ have degree one.
The same description of the automorphism group also shows that $A_q$ admits no nonzero locally nilpotent derivations. Indeed, suppose that $D\in\LND(A_q)$. Since $D$ is locally nilpotent, the exponential
\[
\exp(sD)=\sum_{n\geq0}\frac{s^nD^n}{n!}
\]
defines an automorphism of $A_q$ for every $s\in\Bbbk$, and the map
\[
\mathbb G_a\longrightarrow\Aut(A_q),
\quad
s\longmapsto\exp(sD),
\]
is a homomorphism of algebraic groups. Since $\mathbb G_a$ is connected and $\exp(0D)=\id$, its image is contained in the identity component of $\Aut(A_q)$, which is a torus. A torus admits no nontrivial algebraic group homomorphism from $\mathbb G_a$. Consequently, $\exp(sD)=\id$, for every $s\in\Bbbk$. Comparing the coefficient of $s$ in
$\exp(sD)(a)=a$, for $a\in A_q$, gives $D(a)=0$. Therefore, $D=0$.

Let $A_q^1(\Bbbk)
=\Bbbk\langle X,Y\mid YX-qXY=1\rangle$, where $q\in\Bbbk^\ast\setminus\{1\}$. This algebra can be realized as the quantum generalized Weyl algebra $A(\Bbbk[h],q,h-1)$. Since the defining polynomial $h-1$ is not a monomial, \cite[Proposition 2.3]{SuarezVivas} gives
\[
\LFD(A_q^1(\Bbbk))=\Bbbk H,
\]
where $H(X)=X$ and $H(Y)=-Y$. Consequently, $\LND(A_q^1(\Bbbk))=\{0\}$, because, for $c\neq0$, $(cH)^n(X)=c^nX\neq0$, for every $n\geq1$. By \cite[Theorem B]{SuarezVivas}, if $q\neq-1$, every
automorphism of $A_q^1(\Bbbk)$ has the form
\[
X\longmapsto aX,
\quad
Y\longmapsto a^{-1}Y,
\quad a\in\Bbbk^\ast.
\]

If $q=-1$, every automorphism is either of this form or of
the form
\[
X\longmapsto aY,
\quad
Y\longmapsto a^{-1}X,
\quad a\in\Bbbk^\ast.
\]
These descriptions also apply when $q$ is a root of unity.
Thus every automorphism of $A_q^1(\Bbbk)$ has degree one.
In particular, for every derivation $D$ of $A_q^1(\Bbbk)$,
the degrees of the automorphisms in
$\Aut_D(A_q^1(\Bbbk))$ are bounded.

Therefore, for either the quantum plane $A=A_q$ or the first
quantum Weyl algebra $A=A_q^1(\Bbbk)$, with
$q\in\Bbbk^\ast\setminus\{1\}$, we have
\[
\LND(A)=\{0\},
\qquad
\deg(\rho)=1
\quad\text{for every }\rho\in\Aut(A).
\]
Consequently, for every nonzero derivation $D$ of $A$,
neither local nilpotence nor the existence of automorphisms
of arbitrarily large degree in $\Aut_D(A)$ occurs.
Hence the equivalence in Pan's criterion holds vacuously
for nonzero derivations of these algebras.
The restriction $D\neq0$ is essential: the zero derivation
is locally nilpotent, whereas
$\Aut_0(A)=\Aut(A)$ has bounded degree.

Finally, together with Pan's theorem \ref{thm_Pan} and
Theorem \ref{main_Ah}, these observations establish the
equivalence for all nonzero derivations of the Ore extensions
of $\Bbbk[x]$, except for the classical first Weyl algebra $A_1$. For $A_1$, recall that Theorem \ref{Pan_A1_LF} establishes the equivalence under the additional assumption that the derivation is locally finite.

\section*{Acknowledgements}
S.\ Lopes was partially supported by CMUP -- Centro de Matem\'atica da Universidade do Porto, member of LASI, which is financed by national funds through FCT -- Funda\c c\~ao para a Ci\^encia e a Tecnologia, I.P., under the project with reference UID/00144/2025, doi: \url{https://doi.org/10.54499/UID/00144/2025}. R.\  Baltazar was partially supported by Rio Grande do Sul Research Foundation -- FAPERGS, project n. 24/2551-0001560-4.


\bibliographystyle{amsplain}
\bibliography{references}

@article{AlevChamarie,
  author  = {Alev, J. and Chamarie, M.},
  title   = {D{\'e}rivations et automorphismes de quelques alg{\`e}bres quantiques},
  journal = {Communications in Algebra},
  volume  = {20},
  number  = {6},
  pages   = {1787--1802},
  year    = {1992},
  doi     = {10.1080/00927879208824431}
}

@article{AlimbaevNaurazbekovaKozybaev,
  author  = {Alimbaev, A. A. and Naurazbekova, A. S. and Kozybaev, D. Kh.},
  title   = {Linearization of automorphisms and triangulation of derivations of free algebras of rank 2},
  journal = {Siberian Electronic Mathematical Reports},
  volume  = {16},
  pages   = {1133--1146},
  year    = {2019},
  doi     = {10.33048/semi.2019.16.077}
}

@article{Re1968,
  author  = {Rentschler, R.},
  title   = {Op{\'e}rations du groupe additif sur le plan affine},
  journal = {C. R. Acad. Sci. Paris S{\'e}r. A--B},
  volume  = {267},
  pages   = {A384--A387},
  year    = {1968}
}

@article{V92,
  author  = {van den Essen, A.},
  title   = {Locally finite and locally nilpotent derivations with applications to polynomial flows and polynomial morphisms},
  journal = {Proceedings of the American Mathematical Society},
  volume  = {116},
  number  = {3},
  pages   = {861--871},
  year    = {1992},
  doi     = {10.1090/S0002-9939-1992-1111440-5}
}

@article{BaltazarPan2021,
  author  = {Baltazar, R. and Pan, I.},
  title   = {On the automorphism group of a polynomial differential ring in two variables},
  journal = {Journal of Algebra},
  volume  = {576},
  pages   = {197--227},
  year    = {2021},
  doi     = {10.1016/j.jalgebra.2021.02.009}
}

@article{SuarezVivas,
  author  = {Su{\'a}rez-Alvarez, M. and Vivas, Q.},
  title   = {Automorphisms and isomorphisms of quantum generalized {Weyl} algebras},
  journal = {Journal of Algebra},
  volume  = {424},
  pages   = {540--552},
  year    = {2015},
  doi     = {10.1016/j.jalgebra.2014.08.045}
}

@misc{BaltazarSilvaMartini,
  author = {Baltazar, R. and Duarte Silva, L. and Martini, G.},
  title  = {On the isotropy of differential {Ore} extensions},
  year   = {2026},
  note   = {arXiv:2604.17161, \url{https://arxiv.org/abs/2604.17161}}
}

@article{BLO2015,
  author  = {Benkart, G. and Lopes, S. A. and Ondrus, M.},
  title   = {A parametric family of subalgebras of the {Weyl} algebra {I}. Structure and automorphisms},
  journal = {Transactions of the American Mathematical Society},
  volume  = {367},
  number  = {3},
  pages   = {1993--2021},
  year    = {2015},
  doi     = {10.1090/S0002-9947-2014-06144-8}
}

@article{CrodeShestakov,
  author  = {Crode, S. D. and Shestakov, I. P.},
  title   = {Locally nilpotent derivations and automorphisms of free associative algebra with two generators},
  journal = {Communications in Algebra},
  volume  = {48},
  number  = {7},
  pages   = {3091--3098},
  year    = {2020},
  doi     = {10.1080/00927872.2020.1729363}
}

@article{Czerniakiewicz1971,
  author  = {Czerniakiewicz, A.},
  title   = {Automorphisms of a free associative algebra of rank 2},
  journal = {Bulletin of the American Mathematical Society},
  volume  = {77},
  number  = {6},
  pages   = {992--994},
  year    = {1971},
  doi     = {10.1090/S0002-9904-1971-12830-6}
}

@article{Dixmier,
  author  = {Dixmier, J.},
  title   = {Sur les alg{\`e}bres de {Weyl}},
  journal = {Bulletin de la Soci{\'e}t{\'e} Math{\'e}matique de France},
  volume  = {96},
  pages   = {209--242},
  year    = {1968}
}

@article{DrenskyMakarLimanov,
  author  = {Drensky, V. and Makar-Limanov, L.},
  title   = {Locally nilpotent derivations of free algebra of rank two},
  journal = {SIGMA},
  volume  = {15},
  pages   = {091},
  note    = {10 pp.},
  year    = {2019},
  doi     = {10.3842/SIGMA.2019.091}
}

@article{ISF2021,
  author  = {Kaygorodov, I. and Lopes, S. A. and Mashurov, F.},
  title   = {Actions of the additive group {$G_a$} on certain noncommutative deformations of the plane},
  journal = {Communications in Mathematics},
  volume  = {29},
  number  = {2},
  pages   = {269--279},
  year    = {2021},
  doi     = {10.2478/cm-2021-0024}
}

@article{AlevDumas,
  author  = {Alev, J. and Dumas, F.},
  title   = {Invariants du corps de {Weyl} sous l'action de groupes finis},
  journal = {Communications in Algebra},
  volume  = {25},
  number  = {5},
  pages   = {1655--1672},
  year    = {1997},
  doi     = {10.1080/00927879708825943}
}

@article{AwamiVanDenBerghVanOystaeyen,
  author  = {Awami, M. and Van den Bergh, M. and Van Oystaeyen, F.},
  title   = {Note on derivations of graded rings and classification of differential polynomial rings},
  journal = {Bulletin de la Soci{\'e}t{\'e} Math{\'e}matique de Belgique. S{\'e}rie A},
  volume  = {40},
  number  = {2},
  pages   = {175--183},
  year    = {1988}
}

@article{Jung,
  author  = {Jung, H. W. E.},
  title   = {{\"U}ber ganze birationale Transformationen der Ebene},
  journal = {Journal f{\"u}r die reine und angewandte Mathematik},
  volume  = {184},
  pages   = {161--174},
  year    = {1942},
  doi     = {10.1515/crll.1942.184.161}
}

@article{MakarLimanov,
  author  = {Makar-Limanov, L. G.},
  title   = {Automorphisms of a free algebra with two generators},
  journal = {Functional Analysis and Its Applications},
  volume  = {4},
  number  = {3},
  pages   = {262--264},
  year    = {1970},
  doi     = {10.1007/BF01075252}
}

@article{Pan2022,
  author  = {Pan, I.},
  title   = {A characterization of local nilpotence for dimension two polynomial derivations},
  journal = {Communications in Algebra},
  volume  = {50},
  number  = {5},
  pages   = {1884--1888},
  year    = {2022},
  doi     = {10.1080/00927872.2021.1992631}
}

@article{vanDerKulk,
  author  = {van der Kulk, W.},
  title   = {On polynomial rings in two variables},
  journal = {Nieuw Archief voor Wiskunde},
  series  = {3},
  volume  = {1},
  pages   = {33--41},
  year    = {1953}
}

\bigskip

\noindent
\textsc{Rene Baltazar}\\
Instituto de Matem\'atica, Estat\'istica e F\'isica, Universidade Federal do Rio Grande -- FURG, Rua Bar\~ao do Cahy, 125, 95500--000, Santo Ant\^onio da Patrulha/RS, Brazil.\\
\texttt{renebaltazar.furg@gmail.com}

\bigskip

\noindent
\textsc{Samuel A. Lopes}\\
CMUP, Departamento de Matem\'atica, Faculdade de Ci\^encias, Universidade do Porto, Rua do Campo Alegre s/n, 4169--007, Porto, Portugal.\\
\texttt{slopes@fc.up.pt}

\bigskip 

\noindent
\noindent
\textsc{Oscar Armando Hernandez Morales}\\
Faculdade de Matem\'atica, Universidade Federal do Par\'a -- UFPA, Rua Augusto Corr\^ea, 01, 66075--110, Bel\'em/PA, Brazil.\\
\texttt{oscar@ufpa.br}

\end{document}